\RequirePackage{fix-cm}
\documentclass[smallextended]{svjour3}       
 \smartqed  
\usepackage{graphicx,subfig}
\usepackage{mathptmx}      
\usepackage[misc]{ifsym}
\usepackage{latexsym}
\usepackage{amsmath,amssymb,amsfonts}
\usepackage{algorithm}
\usepackage{algorithmic}
\usepackage{geometry}
\usepackage{url}
\usepackage{color}
\usepackage{chngcntr}
\usepackage{enumerate}
\usepackage{animate}

\usepackage[colorlinks=true, linkcolor=blue, urlcolor=blue, citecolor = blue]{hyperref}
\newcommand{\R}{\mathbb{R}}
\newcommand{\Rn}{\mathbb{R}^{n}}
\newcommand{\dom}{\mathop{\rm dom}\nolimits}

\counterwithin{definition}{section}
\counterwithin{figure}{section}
\counterwithin{theorem}{section}
\counterwithin{example}{section}
\counterwithin{remark}{section}
\counterwithin{table}{section}
\counterwithin{table}{section}
\counterwithin{lemma}{section}
\counterwithin{proposition}{section}

\journalname{Computational Optimization  and Applications}
\begin{document}

\title{Visualization of the $\epsilon$-Subdifferential of Piecewise Linear-Quadratic Functions
}


\author{Anuj Bajaj \and Warren Hare \and Yves Lucet}


\institute{A. Bajaj \at Mathematics, 1132 Faculty/Administration Building, Wayne State University, 42 W. Warren Ave., Detroit, MI 48202, USA\\
              \email{{\color{blue}anuj.bajaj@alumni.ubc.ca}}         
           \and
           W. Hare \at
           Mathematics, ASC 353, The University of British Columbia - Okanagan (UBCO), 3187 University Way, Kelowna, BC, V1V 1V7, Canada \\
           \email{{\color{blue}warren.hare@ubc.ca}}
           \and
           Y. Lucet \at
           Computer Science, ASC 350, UBCO \\
           \email{{\color{blue}yves.lucet@ubc.ca}}
           }

\date{Received: date / Accepted: date}

\maketitle

\begin{abstract}
The final publication is available at Springer via \url{http://dx.doi.org/10.1007/s10589-017-9892-y}

Computing explicitly the $\epsilon$-subdifferential of a proper function amounts to computing the level set of a convex function namely the conjugate minus a linear function. The resulting theoretical algorithm is applied to the the class of (convex univariate) piecewise linear-quadratic functions for which existing numerical libraries allow practical computations. We visualize the results in a primal, dual, and subdifferential views through several numerical examples. We also provide a visualization of the Br\o{}ndsted-Rockafellar Theorem.

\end{abstract}

\section{Introduction}
\label{intro}
Subdifferentials generalize the derivatives to nonsmooth functions, which makes them one of the most useful instruments in nonsmooth optimization. The $\epsilon$-subdifferentials, which are a certain relaxation of true subdifferentials, arise naturally in cutting-plane and bundle algorithms and help overcome some limitations of subdifferential calculus of convex functions.  As such, they are a useful tool in convex analysis.

We begin by defining the central concept of this work, the $\epsilon$-subdifferential of a convex function.
\begin{definition}
	\label{Def:2}
	Let $f:\Rn \rightarrow \R \cup \{-\infty,+\infty\}$ be convex, $x \in \text{dom}(f)$ and $\epsilon \geq 0$. The \emph{$\epsilon$-subdifferential} of $f$ at $x$ is the set
	\begin{equation*}
	\partial_{\epsilon}{f(x)} = \{s \in \Rn:f(y) \geq f(x) + \langle s,y - x\rangle - \epsilon \quad \text{for all} \quad  y \in \Rn\}.
	\end{equation*}
The elements of $\partial_{\epsilon}{f(x)}$ are known as the \emph{$\epsilon$-subgradients} of $f$ at $x$.  We define  $\partial_{\epsilon}{f(x)} = \emptyset$ when $ x\notin \text{dom}(f)$. (Historically, \cite{BRONDSTED-65} used the term ``approximate subgradients'', but we adopt the more common terminology of $\epsilon$-subgradient to make the distinction with the approximate subdifferential introduced in~\cite{IOFFE-84}.)
\end{definition}

In Figure \ref{subdifferentials}, we visualize the $\epsilon$-subdifferential of a convex function for $\epsilon = 2$ and $\epsilon  = 0$.  The $\epsilon$-subdifferential is the set of all vectors that create linearizations passing through $f(x) - \epsilon$ that remain under $f$.

   \begin{figure}[H]
   \centering
	\subfloat{\includegraphics[width=0.5\textwidth]{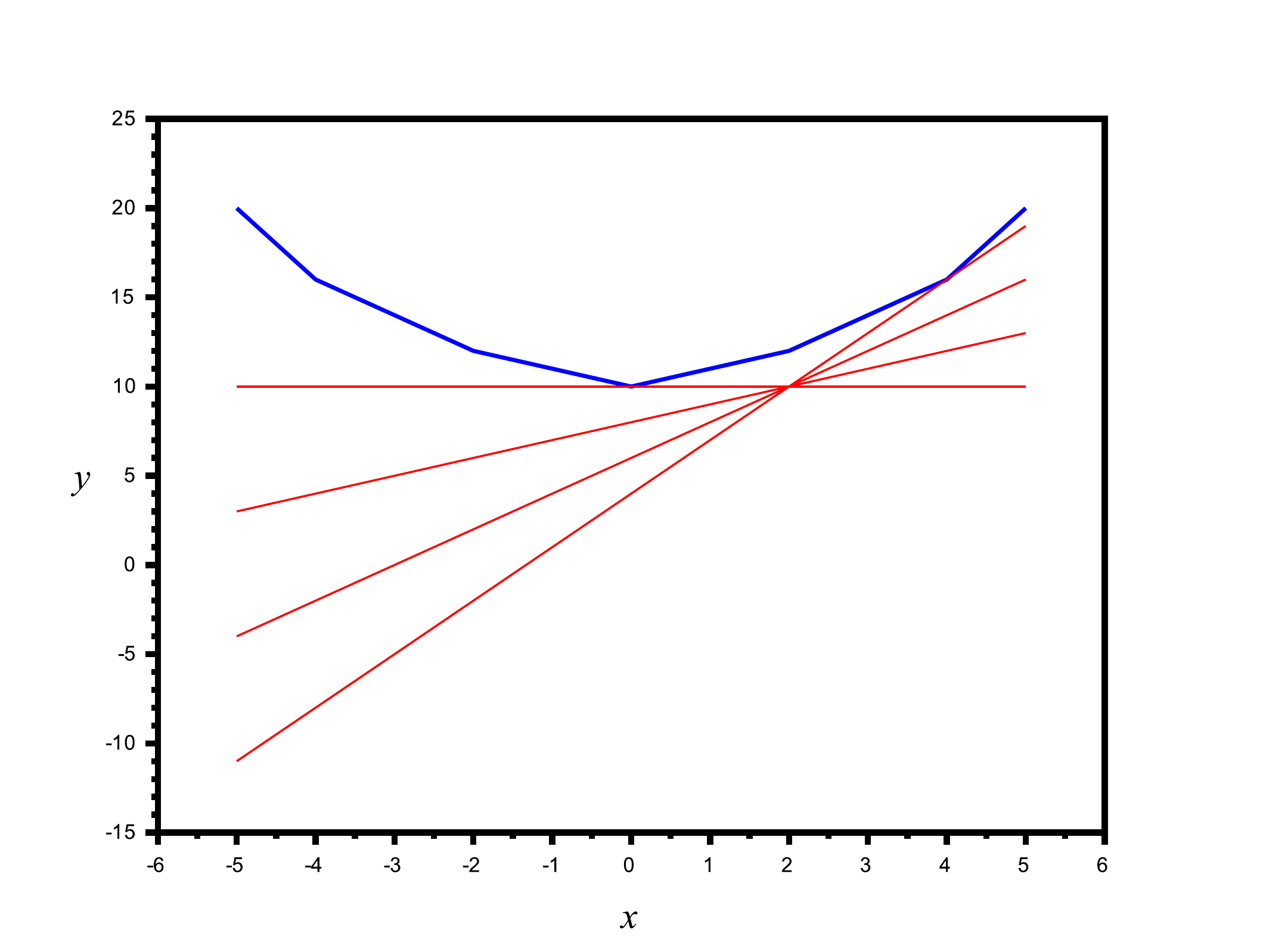}}
	\subfloat{\includegraphics[width=0.47\textwidth]{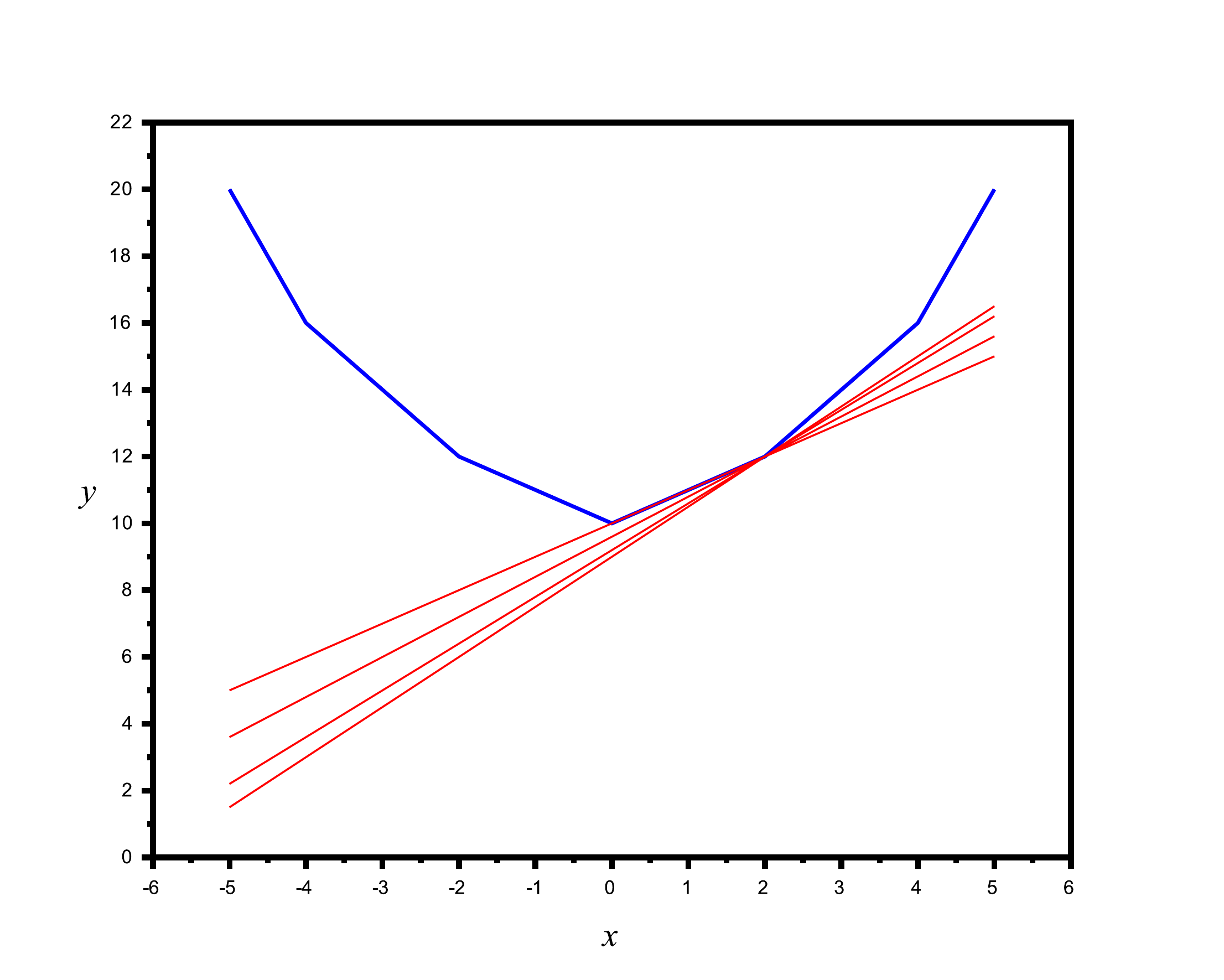}}
	\caption{An illustration of the construction of the $\epsilon$-subdifferential of an example convex function at $x=2$ for $\epsilon = 2$ (left) and $\epsilon = 0$ (right).}
	\label{subdifferentials}
   \end{figure}

Notice that, for $\epsilon = 0$ we obtain the classical subdifferential of a convex function
    $$\partial {f(x)} = \{s \in \Rn:f(y) \geq f(x) + \langle s,y - x\rangle \quad \text{for all} \quad  y \in \Rn\}.$$
It immediately follows from the definition that $ \partial{f(x)} \subseteq \partial_{\epsilon}{f(x)}$ for any $\epsilon \geq 0$.  Thus, the $\epsilon$-subdifferential can be regarded as an enlargement of the true subdifferential.

In the context of nonsmooth optimization, various numerical methods have been developed based on the notion of subdifferentials. One group of foundational methods of particular interest to this work are {\em Cutting Planes} methods.  Cutting Planes methods work by approximating the objective function by a piecewise linear model based on function values and subgradient vectors:
   \begin{equation}
   \label{cpmodel}
   \check{f}_{m}(x) = \max_{i=1,2,\ldots,m}\{f(x^{i}) + \langle s^{i}, x-x^{i}\rangle\}
   \end{equation}
where $s^{i} \in \partial f(x^{i}) = \partial_0 f(x^i)$.  This model is then used to guide the selection of the next iterate.  Various methods based on cutting planes models exist.  For example, proximal bundle methods \cite{Frangioni2002,HareSagastizabal2010,Kiwiel1990,LemarechalSagastizabal1997}, level bundle methods \cite{CorreaLemarechal1993,OliveiraSagastizabal2014,Kiwiel1995}, and hybrid approaches \cite{OliveiraSolodov2016} (among many more).

In Figure \ref{cutting plane} we illustrate 2 iterations of a very basic cutting planes method.

   \begin{figure}[H]
   \centering
   	\centering
   	\subfloat{\includegraphics[width=0.5\textwidth]{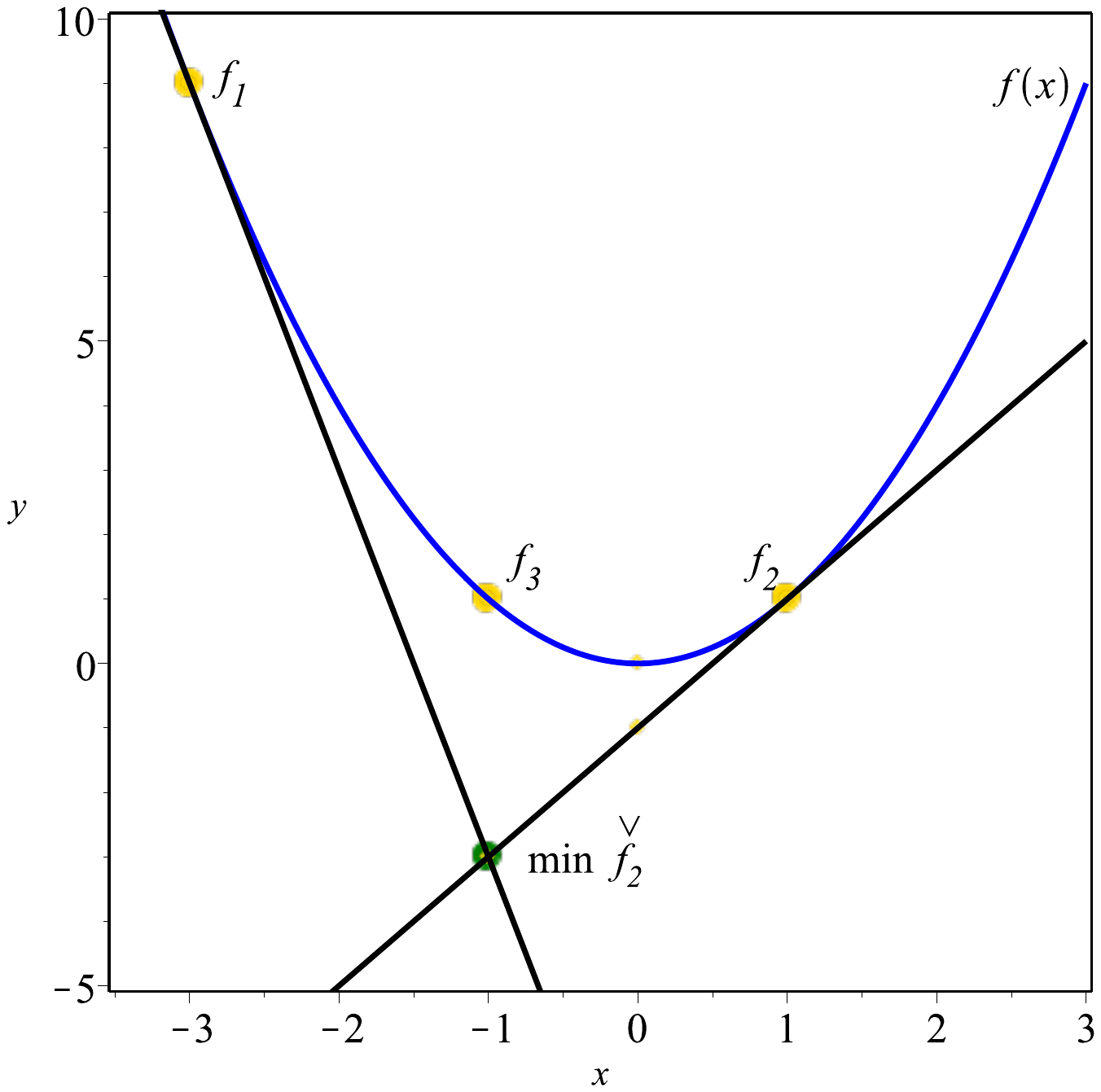}\label{fig:f3}}
   	\subfloat{\includegraphics[width=0.5\textwidth]{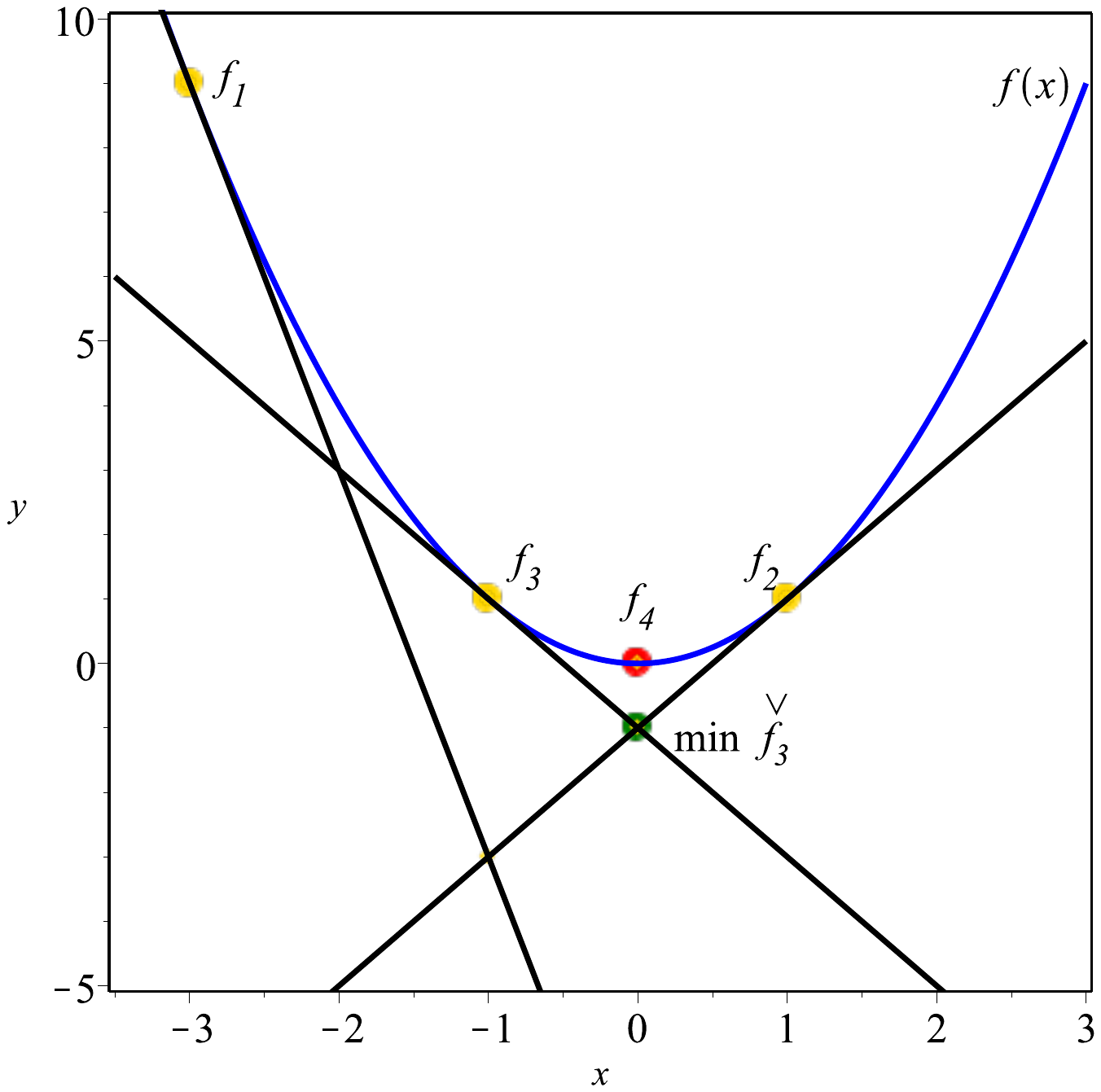}\label{fig:f4}}
   	\vspace{-120pt}
   	\caption{An illustration of a simple Cutting Planes method}
   	\label{cutting plane}
   \end{figure}

In Figure \ref{cutting plane}, we begin with points $x^{1}$ and $x^{2}$ whose function values and (sub)gradients are used to build the model $\check{f}_{2}$. The next iterate, $x^{3}$, is the minimizer of the model  $\check{f}_{2}$, and the function value and a (sub)gradient at $x^3$ is used to refine the model and create $\check{f}_3$.

While this very basic method is generally considered ineffective \cite[Example 8.1]{Bonnans2006}, it has lead to the plethora of methods mentioned above, and helps provide insight on how the $\epsilon$-subdifferential arises naturally in nonsmooth optimization. Specifically, suppose model $\check{f}_k$ is constructed via equation \eqref{cpmodel} and used to select a new iterate $x^{k+1}$ via the simple rule $x^{k+1} \in \arg\min \check{f}_k$.  By equation \eqref{cpmodel} and the definition of the subdifferential, we have $\check{f}_k(x) \leq f(x)$ for all $x$.  Thus,
   \begin{equation*}
   \check{f}_{k}(x^{k+1}) \leq \check{f}_{k}(x) \leq f(x) \quad \text{for all} \quad x,
   \end{equation*}
which yields
\begin{equation}
   \label{epsequation}
   f(x^{k+1}) + \langle 0,x-x^{k+1}\rangle - (f(x^{k+1}) - \check{f}_{k}(x^{k+1})) \leq f(x) \quad \text{for all} \quad x \in \R.
\end{equation}
That is,
\begin{equation*}
   0 \in \partial_{\epsilon_k}f(x^{k+1}) \quad \text{for} \quad \epsilon_k = f(x^{k+1}) - \check{f}_{k}(x^{k+1}) \geq 0.
\end{equation*}
This insight can lead to a proof of convergence (by proving $\epsilon_k \rightarrow 0$) and provides stopping criterion for the algorithm.  When the simple rule $x^{k+1} \in \arg\min \check{f}_k$ is replaced by more advanced methods, convergence analysis often follows a similar path, first showing $0 \in \partial_{\epsilon_k}f(x^{k+1})$ for some appropriate choice of $\epsilon_k$ and then showing $\epsilon_k \rightarrow 0$.  Thus we see one example of the $\epsilon$-subdifferentials role in nonsmooth optimization.

The $\epsilon$-subdifferential has also been studied directly, and a number of calculus rules have been developed to help understand its behaviour \cite{hiriart1995subdifferential,CorreaHantouteJourani2016}.  In this work, we are interested in the development of tools to help compute and visualize the $\epsilon$-subdifferential, at least in some situations.  We feel that such tools will be of great value to build intuition and broader understanding of this important object in nonsmooth optimization.

In this paper, we focus on finding the $\epsilon$-subdifferential of univariate convex piecewise-linear quadratic (plq) function.  Such functions are of interest since they are computationally tractable \cite{GARDINER-13,gardiner2010convex,GARDINER-11,lucet2009piecewise,trienis2007computational} (see also Section \ref{Sec:3}), arise naturally as penalty functions in regularization problems \cite{JMLR:v14:aravkin13a}, and arise in variety of other situations  \cite{JMLR:v14:aravkin13a,dembo1990efficient,hare2014thresholds,rantzer2000piecewise,rockafellar1988essential,rockafellar1986lagrangian}. Moreover, any convex function can be approximated by such a convex plq function.

The present work is organized as follows. Section \ref{Sec:2} provides some key definitions relevant to this work. Subsection \ref{Subsec: 3.1} presents a general algorithm for computing the $\epsilon$-subdifferential of any proper convex function along with a few numerical examples. Subsection \ref{Sec: 3.2} presents an implementation of the general algorithm for the class of univariate convex plq functions.  It also discusses the data structure and the complexity of the algorithm. Section \ref{Sec: 4} illustrates the implementation with some numerical examples, including a visualization of the classic  Br\o{}ndsted-Rockafellar Theorem. Section \ref{Sec: 5} summarizes the work we have done and contains a discussion on the limitations of extending the required implementation. It also provides some directions for future work.

   \section{Key Definitions}
   \label{Sec:2}
   In this section, we provide few key definitions required to understand this work. We assume the reader is familiar with basic definitions and results in convex analysis.

    \begin{definition}
	Given a function $f: \R^{n} \rightarrow \R\cup\{+\infty\}$ (not necessarily convex), the \emph{convex conjugate} (commonly known as the Fenchel Conjugate) of $f$ denoted by $f^{\ast}$ is defined as
	\begin{equation*}
	f^{\ast}(s) = \sup_{x \in \Rn}\{\langle s,x \rangle - f(x)\}.
	\end{equation*}
	We denote $\dom f=\{x\in \Rn : f(x) \in \R \}$.
    \end{definition}

    \begin{definition}
	A set $S \subseteq \Rn$ is called \emph{polyhedral} if it can be specified as finitely many linear constraints. 
	\begin{equation*}
	S = \{x : \langle a_{i},x \rangle \leq b_{i} \hspace{4pt}, \hspace{4pt} i = 1,2,\hdots,p \}
	\end{equation*}
	where for $i = 1,2,\hdots,p$, $a_{i} \in \Rn$ and $b_{i} \in \R$.
    \end{definition}

     \begin{definition}
     A function $f:\Rn \rightarrow \R \cup \{-\infty,+\infty\}$ is \emph{piecewise linear-quadratic} (plq) if $\text{dom}(f)$ can be represented as the union of finitely many polyhedral sets, relative to each of which $f(x) = \dfrac{1}{2}\langle Ax,x \rangle + \langle b,x \rangle + c$ where $A \in \R^{n \times n}$ is a symmetric matrix, $b \in \R^{n}$ and $c \in \R$.\\
     Note that a plq function is continuous on its domain.
    \end{definition}

    \section{Algorithmic Computation of the \texorpdfstring{$\epsilon$}{eps}-subdifferential}
    \label{Sec:3}
    In this section, we propose a general algorithm that enables us to compute the $\epsilon$-subdifferential for any proper function. While the algorithm would be difficult (or impossible) to implement in a general setting, we shall present an implementation specifically for univariate convex plq functions (Section \ref{Sec: 3.2}). We then illustrate the implementation with some numerical examples (Section \ref{Sec: 4}).

    \subsection{The Appx\_Subdiff Algorithm}
    \label{Subsec: 3.1}
We now prove elementary results that will justify the algorithm. Note that the function $m$ defined next is only introduced because it is already available in the CCA numerical library; it is not necessary from a theoretical viewpoint.

      \begin{proposition}
      	\label{thm: 3.1}
      	Let $f : \Rn \rightarrow \R \cup \{+\infty\}$ be a proper function, $\bar{x} \in \mathrm{dom}(f)$ and $\epsilon > 0$.  Note $l_{\bar{x}}: s \mapsto \epsilon-f(\bar{x}) + \langle s,\bar{x} \rangle$ and $m(s) = \min\{f^{\ast}(s),l_{\bar{x}}(s)\}$.  Then
      	\begin{equation}
      	\partial_{\epsilon}f(\bar{x}) = \{s \in \mathbb{R}^{n}: m(s) = f^{\ast}(s)\}.
      	\end{equation}
       \end{proposition}

       \begin{proof}
			Applying the definition of $m$, $l_{\bar{x}}$, and $f^*$ we obtain
			\begin{align*}
			\{s \in \mathbb{R}^{n}: m(s) = f^{\ast}(s)\}
					&= \{s \in \mathbb{R}^{n}:  f^{\ast}(s) \leq l_{\bar{x}}(s) \}, \\
					&= \{s \in \mathbb{R}^{n}:  \langle s, x \rangle -f(x) \leq \epsilon -f(\bar{x}) +\langle s, \bar{x} \rangle, \; \forall x \}, \\
					&= \partial_{\epsilon}f(\bar{x}).
						\end{align*}
      \qed \end{proof}
       Applying Proposition \ref{thm: 3.1} immediately produces the following algorithm for computing the $\epsilon$-subdifferential of a proper function.

   \begin{algorithm}[H]
   \normalsize
   \caption{Appx\_Subdiff Algorithm}
   \textbf{Input}: $f$ (proper function), $\bar{x} \in \text{dom}(f)$, $\epsilon > 0$\\
   \textbf{Output}: $X$
   \begin{algorithmic}[1]
    \STATE Compute $f^{\ast}(s)$ \\
   \STATE Define $l_{\bar{x}}(s) = \epsilon - f(\bar{x}) + \langle s,\bar{x} \rangle$ \\
    \STATE Define $m(s)$ = $\min\{f^{\ast}(s),l_{\bar{x}}(s)\}$ \\
    \STATE Output $X = \{s \in \text{dom}(f^{\ast}): m(s) = f^{\ast}(s)\}$ \\
   \end{algorithmic}
   \label{alg1}
   \end{algorithm}
   To shed some light upon the algorithm we consider the following example.
   \begin{example} Consider the function $f(x) = \dfrac{\| x\|^{p}}{p}$, $x \in \Rn$ where $1<p<\infty$.
    From \cite[Table 3.1]{borwein2010convex} we have $ f^{\ast}(s) = \dfrac{\| s\|^{q}}{q}$, $s \in \Rn$ where $ \dfrac{1}{p} + \dfrac{1}{q} = 1$.
    We also have $l_{\bar{x}}(s)= \epsilon -  \dfrac{\| \bar{x}\|^{p}}{p}  + \langle s,\bar{x} \rangle $.
   Thus, we have
    \begin{eqnarray}
    \label{equation 4.1}
    \partial_{\epsilon}f(\bar{x})
     &=&  \{s \in \Rn: f^{\ast}(s) \leq l_{\bar{x}}(s) \} \nonumber \\
    &=&  \bigg \{ s \in \Rn: \dfrac{\|s\|^{q}}{q} \leq \epsilon -  \dfrac{\| \bar{x}\|^{p}}{p}  + \langle s,\bar{x} \rangle \bigg\}.
    \end{eqnarray}
   In particular, for $\bar{x} =0$, Equation (\ref{equation 4.1}) becomes
    \begin{eqnarray*}
     \partial_{\epsilon}f(0)
      &=&  \bigg \{ s \in \Rn: \dfrac{\| s\|^{q}}{q} \leq \epsilon \bigg\}\\
      &=& \{ s \in \Rn: \| s\|^{q} \leq q\epsilon \}\\
     &=& \bigg\{ s \in \Rn: \| s\| \leq \bigg(\dfrac{p\epsilon}{p-1}\bigg)^{(p-1)/p} \bigg\} .
     \end{eqnarray*}
     \label{Example4.1}

   The particular case of $p = 5$ and $\epsilon =1$ is illustrated in Figure \ref{Fig1}.
   \begin{figure}[H]
   \begin{minipage}[b]{0.35\linewidth}
   	\includegraphics[scale=0.5]{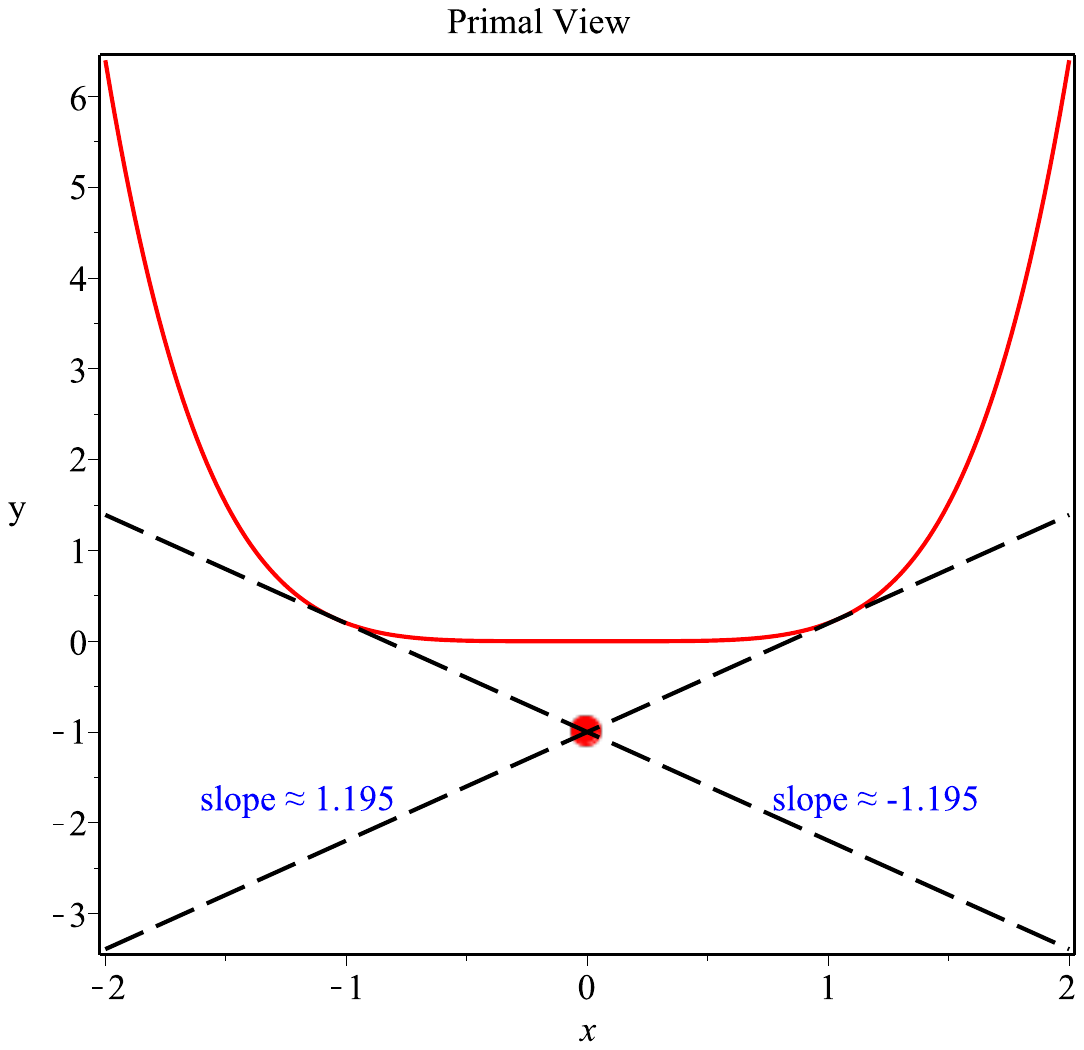}
   \end{minipage}
   \hspace{30pt}
   \begin{minipage}[b]{0.35\linewidth}
   	\includegraphics[scale=0.5]{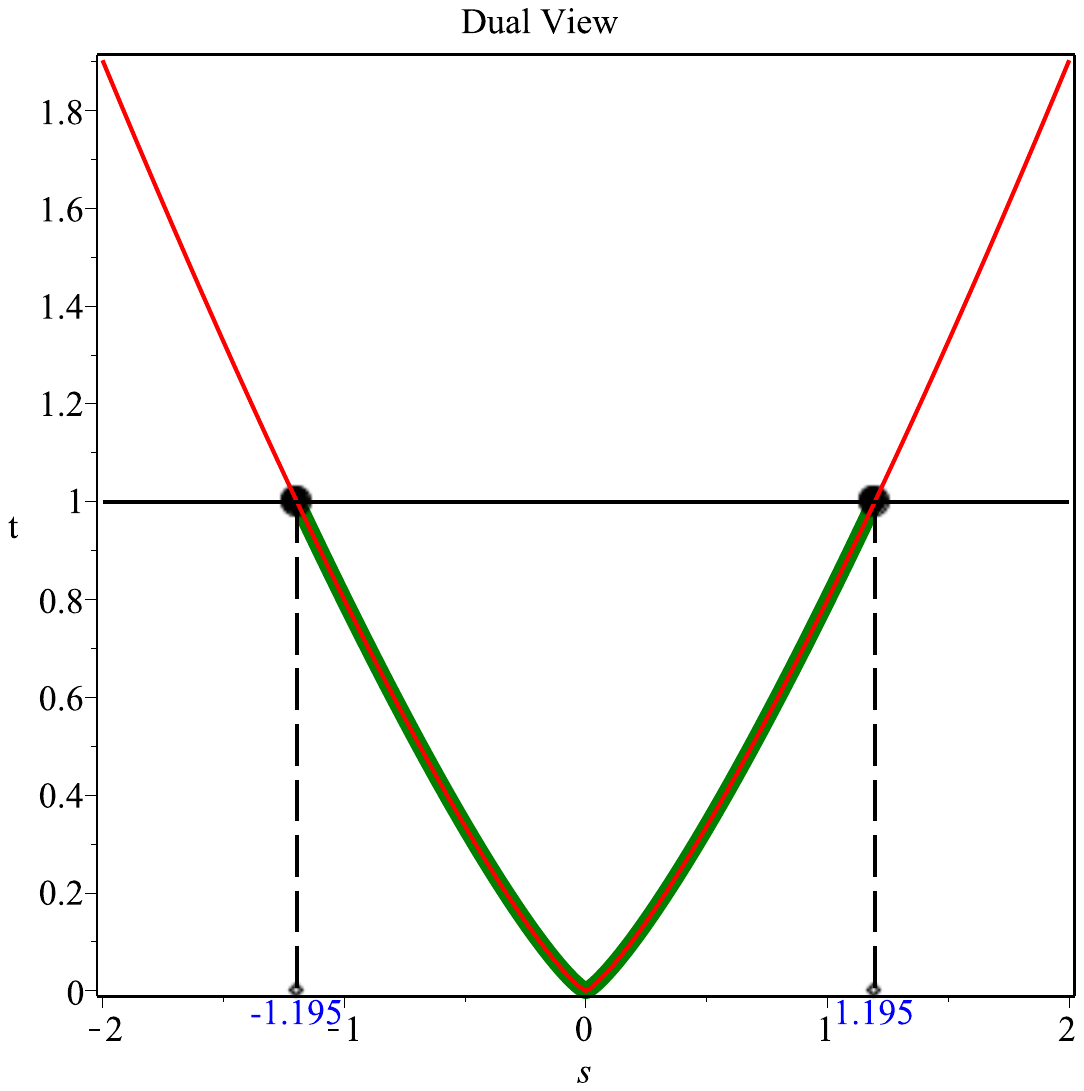}
   \end{minipage}
   \vspace{-210pt}
   \caption{The primal view (left) depicts the graph of $f(x)= \frac{\lvert x\rvert^{5}}{5}$ (red curve) along with the black dashed lines passing through the point $(\bar{x},f(\bar{x})-\epsilon) = (0,-1)$ (red dot) having slopes $-1.195$ and $1.195$ respectively (the lower and upper bounds of $\partial_{\epsilon} f(\bar{x})$). The dual view depicts the graphs of $f^{\ast}(s) =  \frac{4}{5}\lvert s\rvert^{5/4}$ (red curve) and $l_{\bar{x}}(s)$ (solid black line). The green curve shows when the graphs of  $m(s)$ and $ f^{\ast}(s)$ coincide.}
   	\label{Fig1}
   \end{figure}
   \end{example}



  Given the framework of Algorithm \ref{alg1}, a natural question to ask is whether there exists a collection of functions which allows for a general implementation. As mentioned, we consider the well-known class in Nonsmooth Analysis of plq functions.

   \subsection{Implementation: Convex univariate plq Functions}
   \label{Sec: 3.2}

   Our goal in this research is to develop a software that computes and visualizes $\partial_{\epsilon}f(\bar{x})$ at an arbitrary point $\bar{x}$ and $\epsilon > 0$ for a proper convex plq function. As visualization is a key goal, we shall focus on univariate functions.

   \begin{remark}
   \label{definition 4.7}
   	Suppose $f: \R \rightarrow \R \cup\{+\infty\} $ is a proper function. Then, $f$ is a plq function if and only if it can be represented in the form
   	\begin{equation}
   	f(x) =  \left\{
   	\begin{array}{ll}
   	q_{0}(x) = a_{0}x^{2} + b_{0}x + c_{0} , &\text{if} \hspace{5pt}  -\infty < x < x_{0} \\\\
   	q_{1}(x) = a_{1}x^{2} + b_{1}x + c_{1}, &\text{if} \hspace{5pt} x_{0} \leq x \leq x_{1}\\\\
   	q_{2}(x) = a_{2}x^{2} + b_{2}x + c_{2}, &\text{if} \hspace{5pt} x_{1} \leq x \leq x_{2}\\\\
   	\vdots & \vdots \\\\
   	q_{N-1}(x) = a_{N-1}x^{2} + b_{N-1}x + c_{N-1},  &\text{if} \hspace{5pt} x_{N-2} \leq x \leq x_{N-1}\\\\
   	q_{N}(x) = a_{N}x^{2} + b_{N}x + c_{N}, &\text{if} \hspace{5pt}  x_{N-1} < x < +\infty,
   	\end{array}\right.
   	\label{eqn4.3}
   	\end{equation}
   	where, $a_{i} \in \R$ for $i= \{0,1,\cdots,N\}$, $b_{i} \in \R$ for $i= \{0,1,\cdots,N\}$, $c_{i} \in \R$ for $i= \{1,\cdots,N-1\}$ and $c_{i} \in \R\cup \{+\infty\}$ for $i= \{0,N\}$.
   	\end{remark}
	

    An interesting property of plq functions is that they are closed under many basic operations in convex analysis: Fenchel conjugation, addition, scalar multiplication, and taking the Moreau envelope \cite[Proposition 5.1]{lucet2009piecewise}.
		

     \begin{remark}
    Even though the minimum of two plq functions is not necessarily a plq function (see Example~\ref{minfunction}), we can still compute the $\epsilon$-subdifferential from the plq data structure explained in Section \ref{Susubsec: 3.2.1}.
     \end{remark}

    \begin{example}    \label{minfunction}
    Consider $f_{1}(x) = 0$ if $x \in [-1,1]$ and $+\infty$ otherwise; and
    $f_{2}(x) =   x$. Clearly $f_{1}$ and $f_{2}$ are proper convex plq functions but $f(x) = \min\{f_{1}(x),f_{2}(x)\} =  x$ if $x < 0$, $0$ if $0 \leq x \leq 1$, and $x$ when $1 < x$. Notice $f$ is discontinuous at $x=1$ as shown by Figure \ref{minfunctionfigure}.
       	
      \begin{figure}[H]
      \centering \vspace{-1.5cm} \hspace{1cm}
      \includegraphics[scale=0.5]{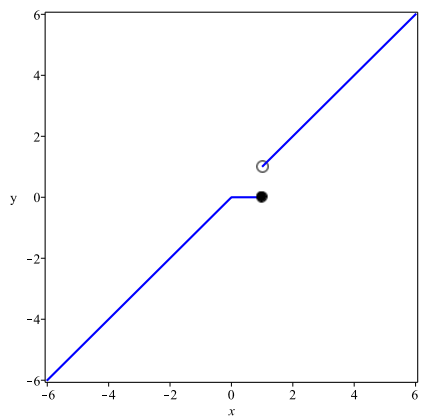}
			\vspace{-7.5cm}
      \caption{The function $f(x)=\min(\iota_{[-1,1]}(x),x)$ is discontinuous at $x=1$.}
      \label{minfunctionfigure}
      \end{figure}
  \end{example}

   	To implement Algorithm \ref{alg1}, for univariate convex plq functions, we shall use the Computational Convex Analysis (CCA) toolbox, which is openly available for download at \cite{atoms}. It is coded using Scilab, a numerical software freely available \cite{SCI}.
   	The toolbox encompasses many algorithms to compute fundamental convex transforms of  univariate plq functions, as introduced in \cite{lucet2009piecewise}. Table \ref{Table4.1} outlines the operations available in the CCA toolbox important to this work.
   	
   	\begin{table}[H]
   	\normalsize
   	\caption{Functions in the CCA Toolbox relevant to Algorithm \ref{alg1}}
   	\centering
    \begin{tabular}{l l}
   	\hline\hline
    \textbf{Function} & \hspace{50pt} \textbf{Description} \\
    \hline
    plq\_check(plq$f$) & \hspace{50pt} Checks integrity of a plq function\\
    plq\_isConvex(plq$f$) & \hspace{50pt} Checks convexity of a plq function \\
    plq\_lft(plq$f$) & \hspace{50pt} Fenchel conjugate of a plq function \\
    plq\_min(plq$f_{1}$,plq$f_{2}$) & \hspace{50pt} Minimum of two plq functions \\
    plq\_isEqual(plq$f_{1}$,plq$f_{2}$) & \hspace{50pt} Checks equality of two plq functions \\
    plq\_eval(plq$f$,$X$) & \hspace{50pt} Evaluates a plq function on the grid $X$\\
   	\hline
    \end{tabular}
   	\label{Table4.1}
   	\end{table}
   	
   	\subsubsection{Data Structure}\label{Susubsec: 3.2.1}
   	We next shed some light on the data structure used in the CCA library. The CCA toolbox stores a plq function as an $(N+1) \times$4 matrix, where each row represents one interval on which the function is quadratic. 
   	
   	For example, the plq function $f : \R \rightarrow \R \cup\{+\infty\} $ defined by \eqref{eqn4.3}
   is stored as
   \begin{equation}
   \text{plq}f = \begin{bmatrix}
   x_{0} &&& a_{0} &&& b_{0} &&& c_{0} \\
   x_{1} &&& a_{1} &&& b_{1} &&& c_{1} \\
   \vdots&&& \vdots &&& \vdots &&& \vdots\\
   x_{N-1} &&& a_{N-1} &&& b_{N-1} &&& c_{N-1}\\
   +\infty &&& a_{N} &&& b_{N} &&& c_{N}
   \end{bmatrix}.
   \label{eqn4.5}
   \end{equation}
   Note that, if $c_{0} = +\infty$ or $c_{N} = +\infty$, then the structure demands that $a_{0} = b_{0} =0$ or $a_{N} = b_{N} =0$ respectively. If $f(x)$ is a simple quadratic function, then $N=0$ and $x_{0} = +\infty$. Finally, the special case of $f(x)$ being a shifted indicator function of a single point $\tilde{x} \in \R$,
   \begin{equation*}
   f(x) = \iota_{\{\tilde{x}\}} + c = \left\{
   \begin{array}{lc}
   c,  & x = \tilde{x}\\
   +\infty, & x \neq \tilde{x}
   \end{array}\right. \end{equation*}
   where $ c \in \mathbb{R}$, is stored as a single row vector plq$f$ = $ \begin{bmatrix}
   \tilde{x} && 0 && 0 && c
   \end{bmatrix}$.

   \begin{remark}
   Throughout this paper, we shall designate $f$ and $\mathrm{plq}f$ for the mathematical function and the corresponding plq matrix representation.
   \end{remark}

  \subsubsection{The plq\_epssub Algorithm}\label{Susubsec: 3.2.2}
  Following the plq data structure we rewrite Algorithm \ref{alg1} for the specific class of univariate convex plq functions. Prior to presenting the algorithm we establish its validity.

   \begin{theorem}
   \label{theorem 4.7}
	Let $f: \R \rightarrow \R \cup \{+\infty\}$ be a univariate convex plq function, $\bar{x} \in \mathrm{dom}(f)$ and $\epsilon > 0$. Let $\partial_{\epsilon}f(\bar{x}) = [v_{l},v_{u}]$. Then one of the following hold.
	\begin{enumerate}
		
	\item If $\mathrm{plq}f^{\ast} = \begin{bmatrix}
	s_{0} &&& \tilde{a}_{0} &&& \tilde{b}_{0} &&& \tilde{c}_{0} \end{bmatrix}$ and $s_{0} \in \R$, then $f^{\ast}(s) = \iota_{\{s_{0}\}}(s) + \tilde{c}_{0}$ and $\tilde{a}_{0}=\tilde{b}_{0}=0$; so $\mathrm{dom}(f^{\ast})=\{s_{0}\}$ and
	\begin{equation*}
	v_{l} =v_{u} = s_{0} .
	\end{equation*}
	In this case, $f$ must be a linear function, i.e., $f(x) = \langle s_{0}, x\rangle - \tilde{c}_{0} $ and $\tilde{c}_{0} \in \R$.
		
	\item	Otherwise, let
	\begin{equation*}
	\mathrm{plq}l =
	\begin{bmatrix}
	+\infty &&& 0 &&& \bar{x} &&& (\epsilon -f(\bar{x}))
	\end{bmatrix}
	\end{equation*}
	and
	\begin{equation*}
	\mathrm{plq}m =
	\begin{bmatrix}
	\hat{s}_{0} &&& \hat{a}_{0} &&& \hat{b}_{0} &&& \hat{c}_{0} \\
	\hat{s}_{1} &&& \hat{a}_{1} &&& \hat{b}_{1} &&& \hat{c}_{1} \\
	\vdots &&& \vdots &&& \vdots &&& \vdots \\
	\hat{s}_{k-1} &&& \hat{a}_{k-1} &&& \hat{b}_{k-1} &&& \hat{c}_{k-1} \\
	\hat{s}_{k} &&& \hat{a}_{k} &&& \hat{b}_{k} &&& \hat{c}_{k} \\
	\end{bmatrix}
	\end{equation*}
	be the respective plq representations of 
	$l_{\bar{x}}(s) = \epsilon - f(\bar{x}) + \langle s, \bar{x} \rangle$
	and 	$	m(s) = \min\{f^{\ast}(s),l_{\bar{x}}(s)\}$.
	Then the following situations hold.
	\begin{enumerate}
	\label{case 2(a)}
	\item If $k=0$, then $\hat{s}_{0} = +\infty$ and
  \begin{equation*}
	v_{l} =
	-\infty, \quad
	 v_{u} =
	 +\infty .
	\end{equation*}
	In this case, $f$ must be the indicator function of $\bar{x}$ plus a constant, i.e., $f(x) = \iota_{\{\bar{x}\}} - \hat{c}_{0} $ and $\hat{c}_{0} \in \R$.
	\item If $k \geq 1$, then $\hat{s}_{0} \in \R, \hat{s}_{k-1} \in \R $,
	\begin{equation*}
	v_{l} =      \left\{
	\begin{array}{ll}
	\hat{s}_{0}, & \mathrm{if} \quad \begin{bmatrix}\hat{a}_{0} &&& \hat{b}_{0} &&& \hat{c}_{0} \end{bmatrix} = \begin{bmatrix}0 &&& \bar{x} &&& (\epsilon - f(\bar{x}))\end{bmatrix} \\
	-\infty, &  \mathrm{otherwise}\\
	\end{array}\right.
	\end{equation*}
	and
	 \begin{equation*}
	v_{u} =      \left\{
	\begin{array}{ll}
	\hat{s}_{k-1}, & \mathrm{if} \quad \begin{bmatrix}\hat{a}_{k} &&& \hat{b}_{k} &&& \hat{c}_{k} \end{bmatrix} = \begin{bmatrix}0 &&& \bar{x} &&& (\epsilon - f(\bar{x}))\end{bmatrix} \\
	+\infty, &  \mathrm{otherwise}\\
	\end{array}\right. .
	\end{equation*}
    \end{enumerate}
	\end{enumerate}
    \end{theorem}

     In order, to prove Theorem \ref{theorem 4.7}, we require the following lemmas.

   \begin{lemma}
   	\label{itm:third}
   	If $f: \Rn \rightarrow \R$ has the form $f= \langle a,x \rangle + b$ where, $a \in \Rn$ and $ b \in \R$, then for $\epsilon \geq 0$ and $\tilde{x} \in \Rn$
   	\begin{equation*}
   	\partial_{\epsilon} f(\tilde{x}) = \partial f(\tilde{x}) = \{\nabla f(\tilde{x})\} =\{a\}.
   	\end{equation*}
   \end{lemma}

   \begin{lemma}
   \label{lemma 4.8}
   Let $f: \R \rightarrow \R \cup \{+\infty\}$ be a proper convex plq function, $\bar{x} \in \dom (f)$ and $\epsilon > 0$. Define $l_{\bar{x}}(s) = \epsilon-f(\bar{x}) + \langle s,\bar{x} \rangle$ and $m(s) = \min\{f^{\ast}(s),l_{\bar{x}}(s)\}$. Then
	\begin{enumerate}
		\item[(i)] There exists $s \in \R$ such that $m(s) < l_{\bar{x}}(s)$. \label{li}
		\item[(ii)] We have \label{lii}
		\begin{equation}
   	\label{Equation 4.8}
   	m  \equiv f^{\ast} \iff f^{\ast} \equiv \langle a,\cdot \rangle + b
   	\end{equation}
   	where $a =\bar{x}$, $b \leq \epsilon - f(\bar{x})$, and $m  \equiv f^{\ast}$ means for all $s$, $m(s)  \equiv f^{\ast}(s)$.
   	In this case, $f$ must be an indicator function, $f \equiv \iota_{\{\bar{x}\}} - b$ for $b \in \R$, and therefore $\partial_{\epsilon}f(\bar{x}) = \R$.
	\end{enumerate}

   \end{lemma}

   \begin{proof}	
	We prove (i) by contradiction. Suppose $m(s) = l_{\bar{x}}(s)$, for all $s \in \R$, i.e. $l_{\bar{x}}(s) \leq f^{\ast}(s)$, for all $s \in \R$. Then using \cite[Theorem 11.1]{rockafellar2009variational} we obtain
	\[
	f(\bar{x}) = \sup_s \{s\bar{x}-f^*(s)\}
\leq \sup_s \{s\bar{x}-l(s)\}
= \sup_s \{ s \bar{x} - \epsilon + f(\bar{x})- s \bar{x} \}
 = f(\bar{x})-\epsilon
	\]
	and that contradiction proves the lemma.

	For (ii), we have
	\begin{align*}
		m \equiv f^*
		& \Leftrightarrow f^* \leq l_{\bar{x}}, \\
		& \Leftrightarrow s x-f(x) \leq \epsilon -f(\bar{x}) + s\bar{x}, \text{ for all $s, x$},\\
		& \Leftrightarrow s\in \partial_\epsilon f(\bar{x}), \text{ for all $s$},\\
		& \Leftrightarrow \partial_\epsilon f(\bar{x}) = \R.
	\end{align*}
	Now assume there is $y\neq x \in \dom f$. Then $\epsilon + f(y) -f(\bar{x}) \geq s(y-\bar{x})$ for all $s$, which is not possible since the left hand-side is bounded and the right one unbounded. Hence, $\dom f$ is a singleton i.e. $f$ is an indicator function. Conversely, if $f$ is an indicator, the equivalence holds. Since the conjugate of the indicator function of a singleton is linear, we further obtain
	\[
		m \equiv f^*
			 \Leftrightarrow f \equiv \iota_{a} - b,
			 \Leftrightarrow f^{\ast} \equiv \langle a,\cdot \rangle + b.
	\]
	The fact that $a =\bar{x}$ is deduced from $g\equiv f^*-l_{\bar{x}} \leq 0$ ($g$ is a convex function defined everywhere and upper bounded, hence a constant~\cite[Corollary 8.6.2]{rockafellar2015convex}). The fact $b \leq \epsilon - f(\bar{x})$ follows similarly.
  \qed \end{proof}

   \begin{remark}
   	\label{remark 4.12}
   	Let $f: \R \rightarrow \R \cup \{+\infty\}$ be a proper convex plq function, $\bar{x} \in \mathrm{dom}(f)$ and $\epsilon > 0$. Let $\mathrm{plq}m$ be the representation of $m(s) = \min\{f^{\ast}(s),l_{\bar{x}}(s)\}$. Suppose
   	\begin{equation*}
   	\mathrm{plq}m =
   	\begin{bmatrix}
   	s_{0} &&& a_{0} &&& b_{0} &&& c_{0} \\
   	s_{1} &&& a_{1} &&& b_{1} &&& c_{1} \\
   	\vdots &&& \vdots &&& \vdots &&& \vdots \\
   	s_{k-1} &&& a_{k-1} &&& b_{k-1} &&& c_{k-1} \\
   	s_{k} &&& a_{k} &&& b_{k} &&& c_{k} \\
   	\end{bmatrix} \quad \mathrm{with} \quad k \geq 1,
   	\end{equation*}
   	where $s_{0} < s_{1},\hdots,s_{k-1} < s_{k} = +\infty$. Then, for any $i \in \{0,1,\hdots,k-1\}$
   	\begin{equation*}
   	\begin{bmatrix}
   	a_{i} &&& b_{i} &&& c_{i}
   	\end{bmatrix}
   	\neq
   	\begin{bmatrix}
   	a_{i+1} &&& b_{i+1} &&& c_{i+1}
   	\end{bmatrix}.
   	\end{equation*}
   This is because the plq\_min function creates the smallest matrix representation of the minimum of two plq functions.  If $[a_{i} ~ b_{i} ~ c_{i}]$ were equal to $[a_{i+1} ~ b_{i+1} ~ c_{i+1}]$, then the $i+1$-th row would be redundant, so not constructed by the plq\_min function.
   \end{remark}

   Now, we turn to the formal proof of Theorem  \ref{theorem 4.7}.

     \begin{proof}[of Theorem \ref{theorem 4.7}]
     	Note that $\bar{x} \in \mathrm{dom}(f)$ and $\epsilon > 0$ so $\partial_{\epsilon}f(\bar{x}) \neq \emptyset$. 

     	    \hspace{-15pt}\textit{Case $1$}
     		Suppose $\mathrm{plq}f^{\ast} = \begin{bmatrix}
     		s_{0} &&& \tilde{a}_{0} &&& \tilde{b}_{0} &&& \tilde{c}_{0} \end{bmatrix}$ with $s_{0} \in \R$.
     		By definition, $f^{\ast}$ is an indicator function, i.e., $f^{\ast}(s) = \iota_{\{s_{0}\}} + \tilde{c}_{0}$. It also follows that $\tilde{a}_0 = \tilde{b}_{0} = 0$ and $\mathrm{dom}(f^{\ast}) = \{s_{0}\}$. This immediately yields $f^{\ast\ast}(x) = \langle s_{0}, x\rangle - \tilde{c}_{0}$. Since $f$ is  a proper convex plq function, we have that $f = f^{\ast\ast}$ \cite[Theorem 11.1]{rockafellar2009variational}, so $f(x)  = \langle s_{0}, x\rangle - \tilde{c}_{0}$
     		and 
$\partial_{\epsilon}f(\bar{x}) = \{s_{0}\}$.
     		Hence, $v_{l} = v_{u} = \{s_{0}\}$.

     	    \hspace{-15pt} \textit{Case $2(a)$}
     		Suppose \begin{equation*}
     		\mathrm{plq}m =
     		\begin{bmatrix}
     		\hat{s}_{0} &&& \hat{a}_{0} &&& \hat{b}_{0} &&& \hat{c}_{0} \\
     		\hat{s}_{1} &&& \hat{a}_{1} &&& \hat{b}_{1} &&& \hat{c}_{1} \\
     		\vdots &&& \vdots &&& \vdots &&& \vdots \\
     		\hat{s}_{k-1} &&& \hat{a}_{k-1} &&& \hat{b}_{k-1} &&& \hat{c}_{k-1} \\
     		\hat{s}_{k} &&& \hat{a}_{k} &&& \hat{b}_{k} &&& \hat{c}_{k} \\
     		\end{bmatrix}
     		\end{equation*} with $k =0$.
     		In this case,
     		\begin{equation*}
     		\mathrm{plq}m =
     		\begin{bmatrix}
     		\hat{s}_{0} &&& \hat{a}_{0} &&& \hat{b}_{0} &&& \hat{c}_{0}
     		\end{bmatrix}
     		\end{equation*}
     		then $\hat{s}_{0} = +\infty$. Indeed, if $\hat{s}_{0} \in \R$, then
     		\begin{equation*}
     		\mathrm{plq}m =
     		\begin{bmatrix}
     		\hat{s}_{0} &&& 0 &&& 0 &&& \hat{c}_{0}
     		\end{bmatrix} (= \iota_{\{\hat{s}_{0}\}} + \hat{c}_{0}),
     		\end{equation*}
     		which cannot happen as $m(s) = \min\{f^{\ast}(s),l_{\bar{x}}(s)\} \leq l_{\bar{x}}(s) < +\infty$.
     		
     		To see $v_{l} = -\infty, v_{u} = +\infty$, note that from Lemma \ref{lemma 4.8} there exists $ \bar{s} \in \R$ such that $m(\bar{s}) < l_{\bar{x}}(\bar{s})$.	
     		This implies
     		\begin{equation*}
     		m(s) = f^{\ast}(s) \quad \text{for all} \quad s,
     		\end{equation*}	
     		as otherwise
     		\begin{equation*}
     		m(s) =      \left\{
     		\begin{array}{ll}
     		f^{\ast}(s),  & s \in X_{f^{\ast}} \\
     		l_{\bar{x}}(s), & s \in X_{l}\\
     		\end{array}\right.
     		\end{equation*}
     		for some  $X_{f^{\ast}}, X_{l} \subseteq \R$ with $X_{l} \neq \emptyset$.	
     		But, then $m(s)$ would be a piecewise function defined on at least two intervals contradicting $ k =0$.	
     		So $m(s) = f^{\ast}(s)$ for all $s$, which by Lemma \ref{lemma 4.8} gives
     		\begin{equation*}
     		\partial_{\epsilon}f(\bar{x}) = \R.
     		\end{equation*}	
     		Therefore, $v_{l} = -\infty$ and $v_{u} +\infty$.
     		Consequently, from Lemma \ref{lemma 4.8}, $f$ must be the indicator function at $\bar{x}$ plus a constant, i.e., $f(x) = \iota_{\{\bar{x}\}} - \hat{c}_{0}$.

         	\hspace{-15pt} \textit{Case $2(b)$}
         	Suppose \begin{equation*}
     		\mathrm{plq}m =
     		\begin{bmatrix}
     		\hat{s}_{0} &&& \hat{a}_{0} &&& \hat{b}_{0} &&& \hat{c}_{0} \\
     		\hat{s}_{1} &&& \hat{a}_{1} &&& \hat{b}_{1} &&& \hat{c}_{1} \\
     		\vdots &&& \vdots && \vdots &&& \vdots \\
     		\hat{s}_{k-1} &&& \hat{a}_{k-1} &&& \hat{b}_{k-1} &&& \hat{c}_{k-1} \\
     		\hat{s}_{k} &&& \hat{a}_{k} &&& \hat{b}_{k} &&& \hat{c}_{k} \\
     		\end{bmatrix}
     		\end{equation*}with $k \geq 1$.
     		By definition, $\hat{s}_{0}, \hat{s}_{k-1} \in \R$. We consider two subcases to prove the formula for $v_l$.
     		
     		\hspace{-15pt} \textit{Subcase $\mathrm{(i)}$}
     			Suppose $\begin{bmatrix}\hat{a}_{0} &&& \hat{b}_{0} &&& \hat{c}_{0} \end{bmatrix} = \begin{bmatrix}0 &&& \bar{x} &&& (\epsilon - f(\bar{x}))\end{bmatrix}$.
     			
     			By definition, $m(s) = l_{\bar{x}}(s)$ for all $ s < \hat{s}_{0}$. From Remark \ref{remark 4.12} we have
     			\begin{equation*}
     			m(s) = f^{\ast}(s) \quad \text {for all} \quad \hat{s}_{0} \leq s \leq \hat{s}_{1}.
     			\end{equation*}
     			Therefore,
     			\begin{equation*}
     			\inf\{v \in \mathrm{dom}({f^{\ast}}) : m(v) = f^{\ast}(v)\} = \hat{s}_{0}.
     			\end{equation*}
     			So $v_{l} = \hat{s}_{0}$.
     		
     	    	\hspace{-15pt} \textit{Subcase $\mathrm{(ii)}$}
     			Suppose $\begin{bmatrix}\hat{a}_{0} &&& \hat{b}_{0} &&& \hat{c}_{0} \end{bmatrix} \neq \begin{bmatrix}0 &&& \bar{x} &&& (\epsilon - f(\bar{x}))\end{bmatrix}$.
     			
     			By definition, we have 
     			\begin{equation*}
     			m(s) = f^{\ast}(s) \quad \text{for all} \quad s < \hat{s}_{0}.
     			\end{equation*}
     			Therefore,
     			\begin{equation*}
     			\inf\{v \in \mathrm{dom}({f^{\ast}}) : m(v) = f^{\ast}(v)\} = -\infty.
     			\end{equation*}
     			So $v_{l} = -\infty$.
     	    	
            	\hspace{-15pt} The formula for $v_u$ can be proven analogously to that of $v_{l}$.
%
%
%
     		
   \qed  \end{proof}

       \begin{algorithm}[ht!]
       	\normalsize
       	\caption{plq\_epssub Algorithm}
       	\textbf{Input}: $\text{plq}f = \begin{bmatrix}
       	x_{0} && a_{0} && b_{0} && c_{0} \\
       	x_{1} && a_{1} && b_{1} && c_{1} \\
       	\vdots&& \vdots && \vdots && \vdots\\
       	x_{N-1} && a_{N-1} && b_{N-1} && c_{N-1}\\
       	+\infty && a_{N} && b_{N} && c_{N}
       	\end{bmatrix}$, $\bar{x}$, $\epsilon > 0$\\ 
       	\textbf{Output}: $\hat{v}_{l}, \hat{v}_{u}$
       	\begin{algorithmic}[1]
       		\STATE Compute plq\_check(plq$f$)\\
       		\textbf{if} false \textbf{return} `the input function is not plq.'
       		\STATE Compute plq\_isConvex(plq$f$) \\
       		\textbf{if} false \textbf{return} `the input function is not convex.'
       		\STATE Compute plq\_eval(plq$f$,$\bar{x}$) \\
       		\textbf{if} $+\infty$, \textbf{return} `$\bar{x}$ is not in the domain of the function.';\\
       		\STATE \textbf{if} $\mathrm{plq}$f$ = \begin{bmatrix}
       		x_{0} && 0 && 0 && c_{0} \end{bmatrix}$ then
       		\textbf{return} $v_{l} = -\infty$,  $v_{u} = +\infty$;\\
       		\STATE Compute plq$f^{\ast}$ = plq\_lft($\text{plq}f$):
       		\begin{equation*}
       		\text{plq}f^{\ast} = \begin{bmatrix}
       		s_{0} && \tilde{a}_{0} && \tilde{b}_{0} && \tilde{c}_{0} \\
       		s_{1} && \tilde{a}_{1} && \tilde{b}_{1} && \tilde{c}_{1} \\
       		\vdots&& \vdots && \vdots && \vdots\\
       		s_{\tilde{N}-1} && \tilde{a}_{\tilde{N}-1} && \tilde{b}_{\tilde{N}-1} && \tilde{c}_{\tilde{N}-1}\\
       		+\infty && \tilde{a}_{\tilde{N}} && \tilde{b}_{\tilde{N}} && \tilde{c}_{\tilde{N}}
       		\end{bmatrix}.
       		\end{equation*}
       		\textbf{if}
       		$\text{plq}f^{\ast} = \begin{bmatrix}
       		s_{0} && \tilde{a}_{0} && \tilde{b}_{0} && \tilde{c}_{0} \end{bmatrix}$ and $s_{0} \in \R$ \textbf{return} $\hat{v}_{l} = s_{0}, \hat{v}_{u} = s_{0}$;
       		\STATE Define 
       		\begin{equation*}
       		\text{plq}l=
       		\begin{bmatrix}
       		+\infty && 0 & \bar{x} && (\epsilon - \text{plq\_eval}(\text{plq}f,\bar{x})) \end{bmatrix}.
       		\end{equation*} \\
       		\STATE Compute plq$m$ = plq\_min$(\text{plq}f^{\ast},\text{plq}l)$:
       		\begin{equation*}
       		\text{plq}m = \begin{bmatrix}
       		\hat{s}_{0} && \hat{a}_{0} && \hat{b}_{0} && \hat{c}_{0} \\
       		\hat{s}_{1} && \hat{a}_{1} && \hat{b}_{1} && \hat{c}_{1} \\
       		\vdots&& \vdots && \vdots && \vdots\\
       		\hat{s}_{k-1} && \hat{a}_{k-1} && \hat{b}_{k-1} && \hat{c}_{k-1}\\
       		+\infty && \hat{a}_{k} && \hat{b}_{k} && \hat{c}_{k}
       		\end{bmatrix},
       		\end{equation*}
       		where $\hat{s}_{0} < \hat{s}_{1}, \hdots,< \hat{s}_{k-1} < \hat{s}_{k} = +\infty.$
       		
       		\STATE Compute $\hat{v}_{l}, \hat{v}_{u}$: \\
       		If $k=0$, \\
       		\begin{center}
       			$\hat{v}_{l} = -\infty, \quad \hat{v}_{u} = +\infty$
       		\end{center}
       		If $k\geq1$,
       		\begin{equation*}
       		\hat{v}_{l} =      \left\{
       		\begin{array}{ll}
       		\hat{s}_{0}, & \mathrm{if} \quad \begin{bmatrix}\hat{a}_{0} && \hat{b}_{0} && \hat{c}_{0} \end{bmatrix} = \begin{bmatrix}0 && \bar{x} && (\epsilon - \text{plq\_eval}(\text{plq}f,\bar{x}))\end{bmatrix} \\
       		-\infty, &  \mathrm{otherwise}\\
       		\end{array}\right.
       		\end{equation*}
       		and
       		\begin{equation*}
       		\hat{v}_{u} =      \left\{
       		\begin{array}{ll}
       		\hat{s}_{k-1}, & \mathrm{if} \quad \begin{bmatrix} \hat{a}_{k} && \hat{b}_{k} && \hat{c}_{k} \end{bmatrix} = \begin{bmatrix}0 && \bar{x} && (\epsilon - \text{plq\_eval}(\text{plq}f,\bar{x}))\end{bmatrix} \\
       		+\infty, &  \mathrm{otherwise}\\
       		\end{array}\right.
       		\end{equation*}
       	\end{algorithmic}	
       	\label{alg2}
       \end{algorithm}

The details of the computation of the $\epsilon$-subdifferential of a univariate convex plq function are presented in Algorithm~\ref{alg2}. Before looking into the complexity of the algorithm, we note a minor difference between the algorithm implementation and Theorem~\ref{theorem 4.7}.

      \begin{remark}
      	In our implementation, the Case \ref{case 2(a)}$(a)$ of Theorem \ref{theorem 4.7} is coded by detecting if $f$ is an indicator function. That is, if $\text{plq}f = \begin{bmatrix}x_{0} &&& 0 &&& 0 &&& c_{0} \end{bmatrix},
      	$ where $x_{0} \in \R$, then $v_{l} = -\infty$ and $v_{u} = +\infty$ (Lemma \ref{lemma 4.8}) without computing plq\_min.
      \end{remark}
			
       \subsubsection{Complexity of Algorithm \ref{alg2}}
       In order to prove the complexity of Algorithm \ref{alg2}, we require the following lemma.

        \begin{lemma}
        	\label{lemma 4.15}
        	If \text{plq}$f$ has (N+1) rows then \text{plq}$f^{\ast}$ has $O(N)$ rows.
        \end{lemma}
        \begin{proof}
        	Since plq\_lft algorithm is developed to independently operate on $(N+1)$ rows \cite[Table 2]{lucet2009piecewise} and has complexity of $O(N)$ \cite[Table 2]{GARDINER-11}, therefore the size of the output plq$f^{\ast}$ cannot exceed $O(N)$.
        \end{proof} \qed

       We now turn to the complexity of Algorithm \ref{alg2}.

        \begin{proposition}
        	\label{prop 4.6}
        	If \text{plq}$f$ has (N+1) rows then Algorithm \ref{alg2} runs in $O(N)$ time and space.
        \end{proposition}

        \begin{proof} Table \ref{Table4.2} summarizes the complexity of the independent subroutines in Algorithm \ref{alg2} as stated in \cite[Table 2]{lucet2009piecewise}, \cite[Table 2]{GARDINER-11} and the function description in Scilab.

      Thus, from Table \ref{Table4.2} and Lemma \ref{lemma 4.15} the result follows.
      \end{proof} \qed

        \begin{table}[H]
        \normalsize
        \caption{Core subroutines in Algorithm \ref{alg2} and their complexity}
       \centering
       \begin{tabular}{l l l }
       	\hline\hline
      	\textbf{Function} & \hspace{10pt} \textbf{Complexity} & \hspace{10pt} \textbf{Variable Description} \\
       	\hline
       	plq\_check(plq$f$) & \hspace{10pt} $O(N)$ & \hspace{10pt} $N =$ number of rows in plq$f$\\
       	plq\_isConvex(plq$f$) & \hspace{10pt} $O(N)$ \\
       	plq\_lft(plq$f$) & \hspace{10pt} $O(N)$   \\
       	plq\_min(plq$f_{1}$,plq$f_{2}$) & \hspace{10pt} $O(N_{1}+N_{2})$ & \hspace{10pt} $N_{1}, N_{2} =$ number of rows in plq$f_{1}$, plq$f_{2}$ \\
       	plq\_eval(plq$f$,$X$) & \hspace{10pt} $O(N+\tilde{k})$ & \hspace{10pt} $\tilde{k} =$ number of points plq$f$ is evaluated at\\
      	\hline	
       \end{tabular}
       	\label{Table4.2}
      	\end{table}
				
      \section{Numerical Examples}
      \label{Sec: 4}
      We now present several examples which demonstrate how Algorithm \ref{alg2} can be used to visualize the $\epsilon$-subdifferential of univariate convex plq functions. The algorithm has been implemented in Scilab \cite{SCI}. 

\subsection{Computing $\partial_{\epsilon}f(\bar{x})$ for fixed $\epsilon > 0$ and varying $\bar{x}$}

      \begin{example}\label{example 4.6}Let
      \begin{equation*}
      f(x)
      =
      \left\{
      \begin{array}{lc}
      x^{2}/2,  & -\infty< x < 0\\
      0, & 0 \leq x < +\infty
      \end{array}\right.,
      \end{equation*} at $\bar{x} =0$ and $\epsilon =1$. In plq format $f$ is stored as
      	\begin{equation*}
      	\text{plq}f = \begin{bmatrix}
      	0 && 1/2 && 0 && 0 \\
      	+\infty && 0 && 0 && 0
      	\end{bmatrix}.
      	\label{eqn4.6}
      	\end{equation*}
      	Using plq\_lft we obtain
      	\begin{equation*}
      	\text{plq}f^{\ast}(s) =
      	\begin{bmatrix}
      	0 && 0.5 && 0 && 0 \\
      	+\infty && 0 && 0 && +\infty
      	\end{bmatrix}
				\text{ corresponding to }
      	f^*(s) =
      	\left\{
      	\begin{array}{lc}
      	s^{2}/2, & s < 0\\
      	+\infty, & 0 \leq s
      	\end{array}\right..
      	\label{eqn4.7}
      	\end{equation*}
      	Here $l_{\bar{x}}(s)= 1-f(0)+\langle 0,s \rangle = 1$, in plq format we have $ \text{plq}l = \begin{bmatrix}+\infty && 0 && 0 && 1 \end{bmatrix}$. Next, we compute plq\_min$(\text{plq}f^{\ast}, \text{plq}l)$, we obtain
      	\begin{equation*}
      	\text{plq\_min}(\text{plq}f^{\ast}, \text{plq}l)
      	= \begin{bmatrix}
      	-1.4142136 &&& 0 &&& 0 &&& 1 \\
      	0 &&& 0.5 &&& 0 &&& 0 \\
      	+\infty &&& 0 &&& 0 &&& 1
      	\end{bmatrix}
      	\bigg(=
      	\left\{
      	\begin{array}{lc}
      	f^{\ast}(s),  & s \in [-1.4142136,0]\\
      	l_{\bar{x}}(s), & s \notin [-1.4142136,0]
      	\end{array}\right.\bigg).
      	\label{eqn4.8}
      	\end{equation*}
      	Hence, we obtain $\partial_{\epsilon} f(\bar{x}) \approx [-1.414,0]$ as visualized in Figure \ref{Fig5}.
      	 \begin{figure}[H]
      	 	\centering
      	 	\includegraphics[scale=0.54]{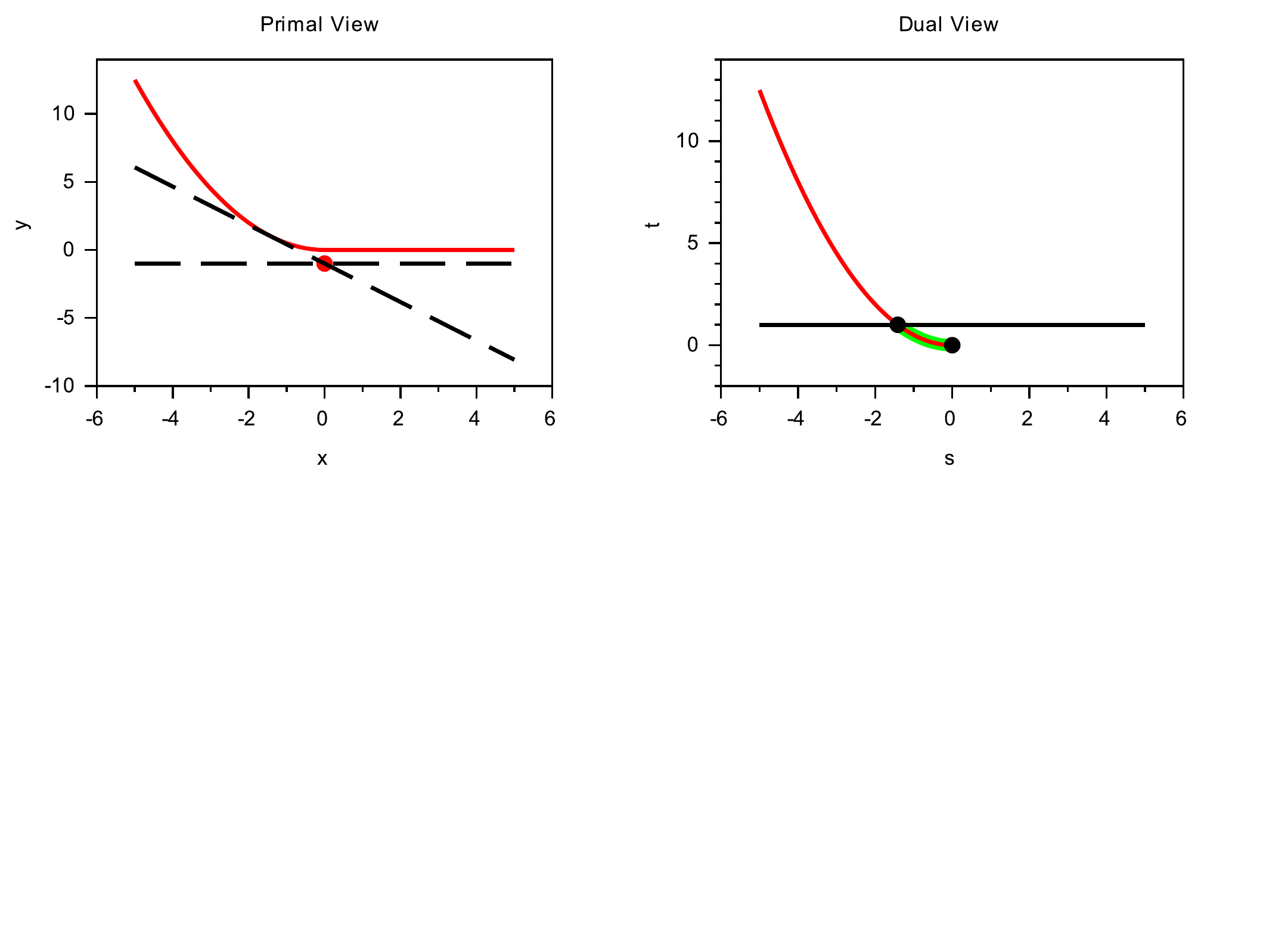}
      	 	\vspace{-130pt}
      	 	\caption{The primal view (left) depicts the graph of $f(x)=x^{2}/2$ if $x < 0$, and $0$ otherwise (red curve) along with the black dashed lines passing through the point $(\bar{x},f(\bar{x})-\epsilon)$ (red dot), with $\bar{x} = 0$ and $\epsilon =1$, having slopes $-1.414$ and $0$ respectively (the lower and upper bounds of $\partial_{\epsilon} f(\bar{x})$). The dual view depicts the graphs of $f^{\ast}(s)$ (red curve) and $l_{\bar{x}}(s)$ (solid black line). The green curve shows when the graphs of  $m(s)$ and $ f^{\ast}(s)$ coincide.}
      	 	\label{Fig5}
      	 \end{figure}
      	\end{example}
      	
      	In one dimension, from Figure \ref{Fig5}, we may geometrically interpret that the $\epsilon$-subdifferential set consists of all possible slopes belonging to the interval $\left[-1.414,0\right]$, resulting in all possible lines with the respective slopes passing through the point ($\bar{x},f(\bar{x})-\epsilon$) = $(0,-1)$. 	

     We now look into visualizing the multifunction $x\mapsto\partial_{\epsilon}f(x)$ for a given $\epsilon > 0$.

   \begin{example} We consider
   	\begin{equation*}
   f(x) = 	\left\{
   	\begin{array}{lc}
   -7x - 5,  & -\infty < x < -1\\
   	x^{2} - x, & -1 \leq x \leq 1 \\
   	2x^{2} - 3x + 1, & 1 < x < \infty
   	\end{array}\right.,
   	\end{equation*} and $\epsilon =1$. As seen in Figure \ref{Fig11a}, for $\bar{x} = -1$, $\partial_{\epsilon}f(\bar{x}) = [-7,-1]$. Correspondingly, for the choice of $\bar{x} = 0.5$ we obtain $\partial_{\epsilon}f(\bar{x}) \approx [-2,2.162]$, as visualized in Figure  \ref{newfigure}.
   	
    \begin{figure}[H]
    \centering
    \includegraphics[scale=0.5]{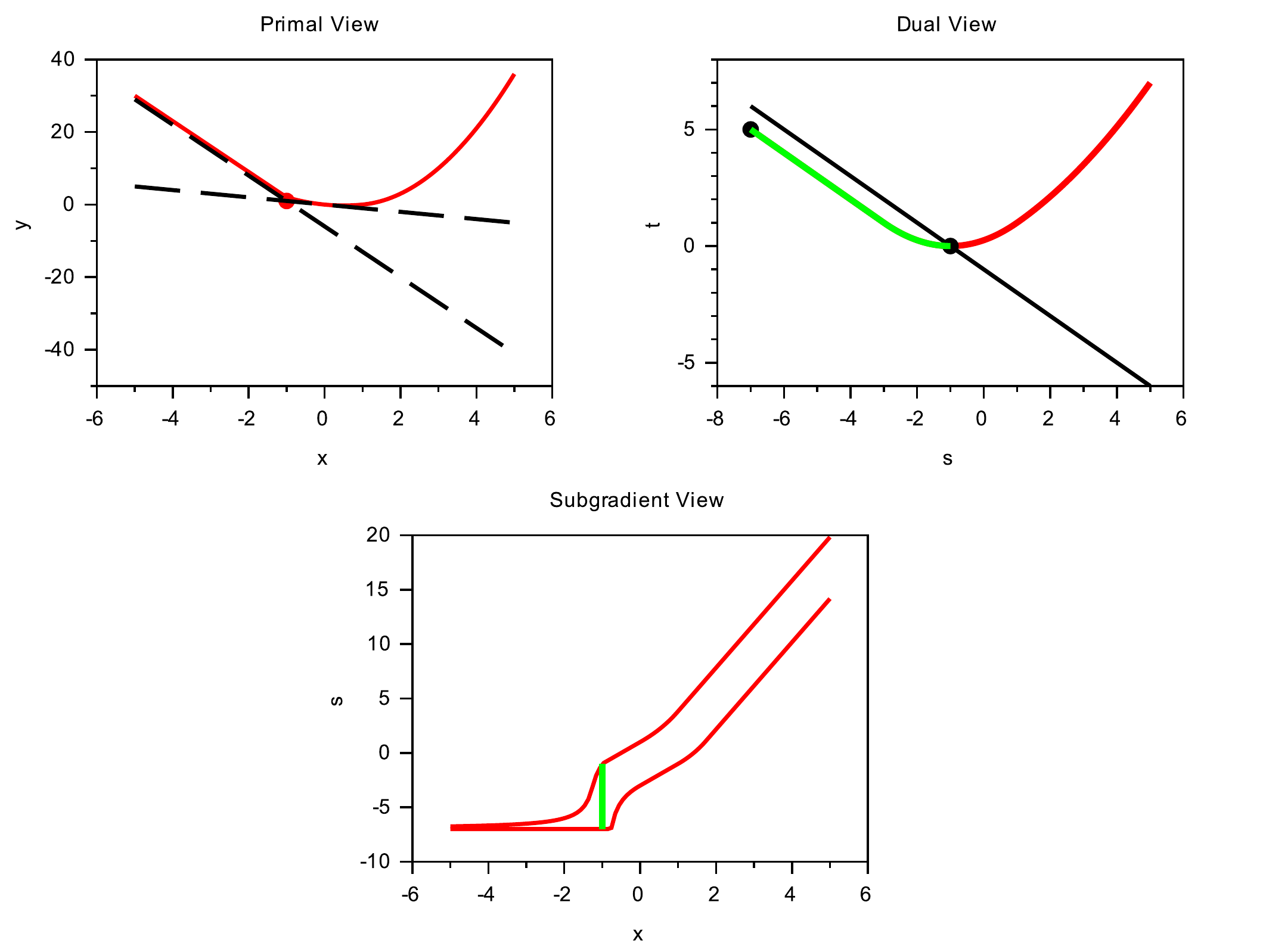}
    \caption{The primal view (left) depicts the graph of $f(x)=-7x-5$ if $x < -1$, $f(x)=x^2-x$ if $-1\leq x \leq 1$ and $2x^2-3x+1$ otherwise (red curve) along with the black dashed lines passing through the point $(\bar{x},f(\bar{x})-\epsilon)$ (red dot), with $\bar{x} = -1$ and $\epsilon =1$, having slopes $-7$ and $-1$ respectively (the lower and upper bounds of $\partial_{\epsilon} f(\bar{x})$). The dual view (right) depicts the graphs of $f^{\ast}(s)$ (red curve) and $l_{\bar{x}}(s)$ (solid black line). The green curve shows when the graphs of  $m(s)$ and $ f^{\ast}(s)$ coincide. The subgradient view (bottom) shows the graph of $x\mapsto \partial_{\epsilon} f(\bar{x})$ (lower and upper bounds) with the green line showing $\partial_{\epsilon}f(-1)$.}
    \label{Fig11a}
    \end{figure}

     \begin{figure}[H]
     	\centering
     	\includegraphics[scale=0.5]{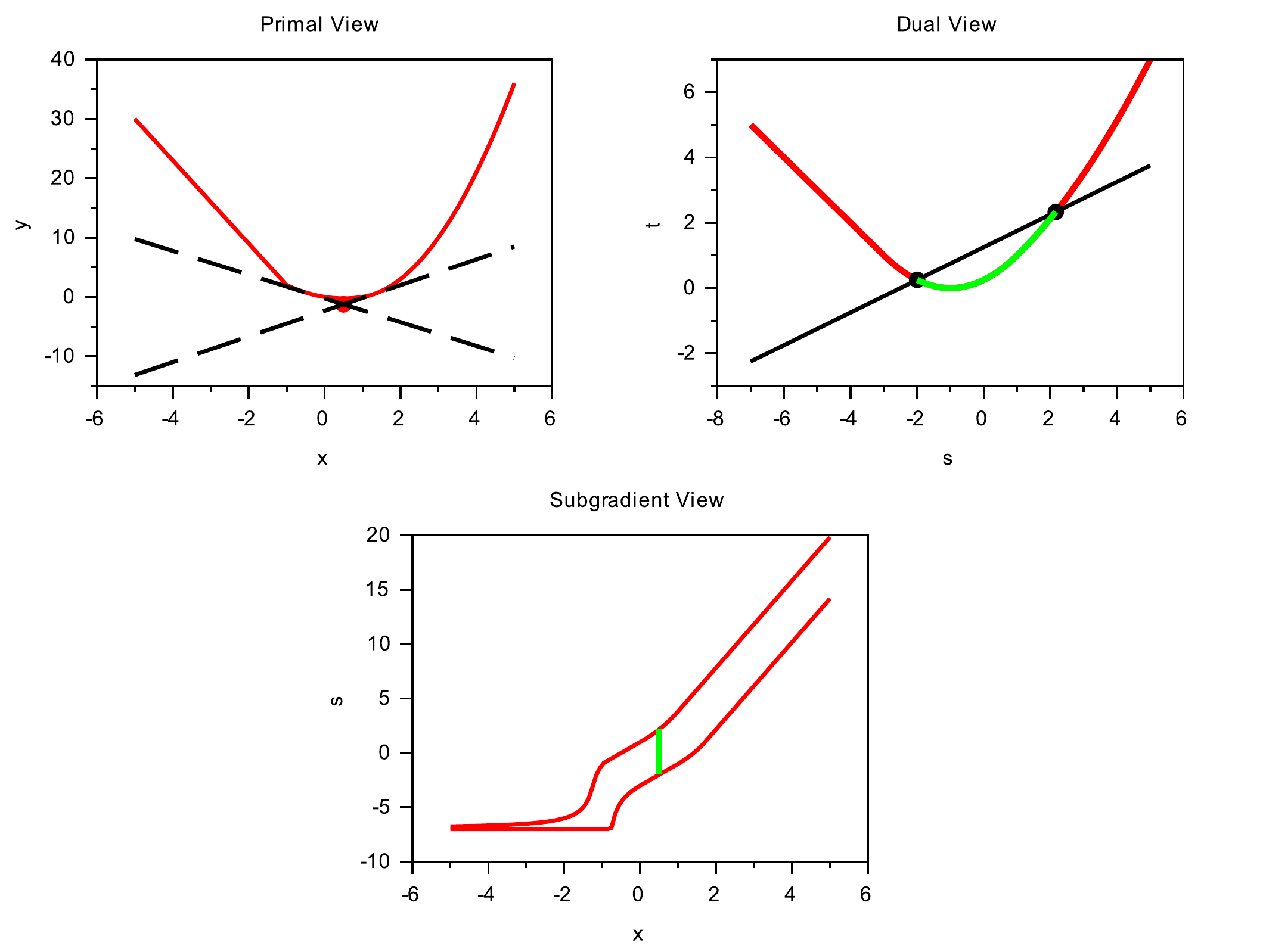}
     	\caption{
The primal view (left) depicts the graph of $f(x)=-7x-5$ if $x < -1$, $f(x)=x2-x$ if $-1\leq x \leq 1$ and $2x^2-3x+1$ otherwise (red curve) along with the black dashed lines passing through the point $(\bar{x},f(\bar{x})-\epsilon)$ (red dot), with $\bar{x} = 0.5$ and $\epsilon=1$, having slopes $-2$ and $2.162$ respectively (the lower and upper bounds of $\partial_{\epsilon} f(\bar{x})$). The dual view (right) depicts the graphs of $f^{\ast}(s)$ (red curve) and $l_{\bar{x}}(s)$ (solid black line). The green curve shows when the graphs of  $m(s)$ and $ f^{\ast}(s)$ coincide. The subgradient view (bottom) shows the graph of $x\mapsto \partial_{\epsilon} f(\bar{x})$ (lower and upper bounds) with the green line showing $\partial_{\epsilon}f(0.5)$.}
     	\label{newfigure}
     \end{figure}
   \end{example}

   In Figure \ref{newfigure}, the graph of $\partial_{\epsilon}f(x)$ (red curve) as a function of $x \in [-5,5]$ with $\epsilon =1$ is sketched by iteratively computing the respective lower and the upper bounds of $\partial_{\epsilon}f(x)$ for 100 equally spaced points in the interval $[-5,5]$. This process takes under $5$ seconds on a basic computer.

   \begin{example}
   Let
   \begin{equation*}
   f(x) = \left\{
   \begin{array}{lc}
   	+\infty,  & x < -6 \\
   	-2x, & -6 \leq x \leq 0\\
   	x^{2}-2x, & 0 < x \leq 2\\
   	2x - 4, & 2 < x \leq 3\\
   	\frac{1}{3}x^{2} -1, & 3 < x
   \end{array}\right..
   \end{equation*}	
   An animated visualization of $\partial_{\epsilon}f(\bar{x})$ for the example is presented in the following Figure \ref{Fig11c}, that takes into account $\epsilon =1$ and the choices of $\bar{x}$ as 50 equally spaced points between $[-5,4.5]$.

    \begin{figure}[ht!]
   \centering
    \animategraphics[autoplay,loop,scale=0.54]{12}{plqanimate4a}{}{}
    \caption{
		Animated version of the primal view (left), dual view (right), and subdifferential graph (bottom) for $f(x)=+\infty$ if $x<-6$, $f(x)=-2x$ if $-6\leq x \leq 0$, $f(x)=x^2-2x$ if $0\leq x \leq 2$, $f(x)=2x-5$ if $2\leq x \leq 3$ and $f(x)=x^2/3-1$ otherwise (red curve) along with the black dashed lines passing through the point $(\bar{x},f(\bar{x})-\epsilon)$ (red dot) for$\epsilon=1$. The slopes are the lower and upper bounds of $\partial_{\epsilon} f(\bar{x})$. The dual view (right) depicts the graphs of $f^{\ast}(s)$ (red curve) and $l_{\bar{x}}(s)$ (solid black line). The green curve shows when the graphs of  $m(s)$ and $ f^{\ast}(s)$ coincide. The subgradient view (bottom) shows the graph of $x\mapsto \partial_{\epsilon} f(\bar{x})$ (lower and upper bounds) with the green line showing $\partial_{\epsilon}f(\bar{x})$.}
     \label{Fig11c}
     \end{figure}
  \end{example}

     \subsection{Computing $\partial_{\epsilon}f(\bar{x})$ for fixed $\bar{x}$ and varying $\epsilon > 0$}
    We can also visualize the graph of $\partial_{\epsilon}f(\bar{x})$  as a function of $\epsilon > 0$ for a given $\bar{x}$.

    \begin{example} \label{Example 4.19}
    Consider  $\bar{x} =0$ and let
   \begin{equation*}
   f(x) =	\left\{
   \begin{array}{lc}
   \frac{1}{6}x^{2} + \frac{1}{3}x,  & x < -2 \\
   x+2, & -2 \leq x \leq 1\\
  	+\infty, & 1 < x
  	\end{array}\right..
    	\end{equation*}
    	
   An animated visualization of $\partial_{\epsilon}f(\bar{x})$ for the example is presented in the following Figure \ref{moveeps}, that takes into account $\bar{x} = -1$ and the choices of $\epsilon > 0$ as 50 equally spaced points between $[0.1,3]$.
    	
   \begin{figure}[ht!]
   \begin{center}
    \animategraphics[autoplay,loop,scale=0.54]{12}{movingeps}{}{}
    \caption{
		Animated version of the primal view (left), dual view (right), and subdifferential graph (bottom) for $f(x)=x^2/6+x/3$ if $x<-2$, $f(x)=x+2$ if $-2\leq x \leq 1$, $f(x)=+\infty$ otherwise (red curve) along with the black dashed lines passing through the point $(\bar{x},f(\bar{x})-\epsilon)$ (red dot) for$\bar{x}=-1$. The slopes are the lower and upper bounds of $\partial_{\epsilon} f(\bar{x})$. The dual view (right) depicts the graphs of $f^{\ast}(s)$ (red curve) and $l_{\bar{x}}(s)$ (solid black line). The green curve shows when the graphs of  $m(s)$ and $ f^{\ast}(s)$ coincide. The subgradient view (bottom) shows the graph of $\epsilon \mapsto \partial_{\epsilon} f(\bar{x})$ (lower and upper bounds) with the green line showing $\partial_{\epsilon}f(-1)$.}
    \label{moveeps}
    \end{center}
    \end{figure}
    \end{example}

\subsection{An illustration of Br\o{}ndsted-Rockafellar Theorem}

In this section, we visualize the Br\o{}ndsted-Rockafellar theorem.

\begin{theorem}\cite[Theorem XI.4.2.1]{hiriart1993convex}
Let $f:\mathbb{R}^{n} \rightarrow \mathbb{R} \cup \{+\infty\}$ be a proper lower- semicontinuous convex function, $\bar{x} \in \mathrm{dom}(f)$ and $\epsilon \geq 0$. For any $\lambda > 0$ and $ s \in \partial_{\epsilon}f(\bar{x})$, there exists $\bar{x}_{\lambda} \in \mathrm{dom}(f)$ and $s_{\lambda} \in \partial f(\bar{x}_{\lambda})$ such that $\lVert \bar{x}_{\lambda}- \bar{x} \rVert \leq \lambda$ and  $\lVert s_{\lambda}-s \rVert \leq \epsilon/\lambda.$
\label{Bron-Roc}
\end{theorem}	

Theorem \ref{Bron-Roc} asserts that for a one-dimensional proper lower-semicontinuous convex function, any $\epsilon$-subgradient at $\bar{x}$ can be approximated by some true subgradient computed (possibly) at some $y\neq \bar{x}$, lying within a rectangle of width $\lambda$ and height $\epsilon/\lambda$. For a better understanding, we consider the following animated example.

\begin{example}
Consider for $x \in \mathbb{R}$
\begin{equation*}
f(x) = \left\{
\begin{array}{lc}
x^{2}/3,  & x \leq -2,\\
x/2+7/3, & -2 \leq x \leq 2.5, \\
x^{2}-8/3,    & 2.5 < x,
\end{array}\right.\end{equation*} and $\bar{x} = -1.5$. Figure \ref{Bron-Roc} visualizes
the Br\o{}ndsted-Rockafellar theorem at $(\bar{x},v_{l}) \approx (-1.5,-1.471) \in \partial_{\epsilon}f(-1.5)$ (black star).

 \begin{figure}[H]
 	\centering
 	\subfloat[For $\epsilon =1$ and $\lambda \in {[0.2,2]}$]{\animategraphics[autoplay,loop,width=0.5\textwidth]{14}{Bron_Rock_animate3}{}{}\label{fig:f1}}
    \hfill
 	\subfloat[For $\lambda = 1$ and $\epsilon \in {[0.1,2]}$]{\animategraphics[autoplay,loop,width=0.5\textwidth]{14}{Bron_Rock_epschange}{}{}\label{fig:f2}}
 	\caption{An illustration of Br\o{}ndsted-Rockafellar theorem at $\bar{x} = -1.5$. The graph depicts $\partial_{\epsilon}f(x)$ (red dashed curve) and $\partial f(x)$ (black curve) for $x \in [-5,5]$. The theorem asserts that the blue rectangle always intersect the black curve.
	}
 	\label{Bron-Rock}
 \end{figure}

In Figure \ref{fig:f1}, for a given $\epsilon=1$, and $(\bar{x},v_{l}) \approx (-1.5,-1.471)$, we plot rectangles having respective dimensions $\lambda \times (\epsilon/\lambda)$ for $50$ different choices of $\lambda \in [0.2,2]$. We observe that, as stated by the Br\o{}ndsted-Rockafellar theorem, for each choice of $\lambda$ the rectangles intersect the true subdifferential. Likewise, in Figure $\ref{fig:f2}$ we repeat the same process with a fixed $\lambda =1$ and 50 different choices of $\epsilon \in [0.1,2]$ leading to a similar conclusion.
\end{example}

   \section{Conclusion and Future Work}
   \label{Sec: 5}
   In this work, we first proposed a general algorithm that computes the $\epsilon$-subdifferential of any proper convex function, and then presented an implementation for univariate convex plq functions.  The implementation allows for rapid computation and visualization of the $\epsilon$-subdifferential for any such function, and extends the CCA numerical toolbox.

   Noting that the algorithm is implementable in one dimension, it is natural to ask whether an extension to higher dimension is possible.  Note that if $f:\R^n \rightarrow \R$, then visualization of the subdifferential is difficult, since $\partial f(x)$ is a set-valued mapping from $\R^n$ into $\R^n$.  In addition, in dimensions greater than $1$, the minimum of two plq functions is no longer representable using the plq data structure presented in Subsection \ref{Susubsec: 3.2.1}.  Consider, for example,
        $$\min(1/2\lVert s \rVert ^{2}-1,0) = \left\{
        \begin{array}{ll}
        1/2\lVert s \rVert ^{2}-1 & \|x\| \leq 1, \\
        0 & \|x\| > 1.
        \end{array}\right.$$
   and note that the domain is not split into polyhedral pieces. Hence, the $\epsilon$-subdifferential is no longer polyhedral even for functions of 2 variables. It is a convex set whose boundary is defined by piecewise curves; in some cases it is an ellipse. 
	
	Note that there has been work on computing the conjugate of bivariate functions~\cite{JAKEE-13,GARDINER-11,GARDINER-13}. However, the resulting data structures are much harder to manipulate. We leave it to future work to extend our results to higher dimensions.

   Two other directions for future work are as follows. First, the current method to produce the subdifferential view (e.g., Figure \ref{Fig11a}) requires computing the $\epsilon$-subdifferential for a wide selection of $x$ values.  It may be possible to improve this through a careful analysis of Proposition \ref{thm: 3.1}.  Another clearly valuable direction of extension would be developing methods to visualize the $\epsilon$-subdifferential of any univariate convex function.  A first approach to this could be achieved by approximating the univariate convex function with a univariate convex plq function.  However, it may be more efficient to try to directly solve $\inf \{ v \in \text{dom}(f^*) : m(v) = f^*(v)\}$ and $\sup \{ v \in \text{dom}(f^*) : m(v) = f^*(v)\}$ using a numerical optimization method.



\begin{acknowledgements}
This work was supported in part by
Discovery Grants \#355571-2013 (Hare) and \#298145-2013 (Lucet) from NSERC, and The University of British Columbia, Okanagan campus. Part of the
research was performed in the Computer-Aided Convex Analysis (CA2) laboratory funded by a
Leaders Opportunity Fund (LOF) from the Canada Foundation for Innovation (CFI) and by a
British Columbia Knowledge Development Fund (BCKDF).

Special thanks to Heinz Bauschke for recommending we visualize the Br\o{}ndsted-Rockafellar Theorem.

The final publication is available at Springer via \url{http://dx.doi.org/10.1007/s10589-017-9892-y}

\end{acknowledgements}

\bibliographystyle{spmpsci}      
\bibliography{mybibliography}   

%
%

\end{document}